\documentclass[12pt]{article}

\usepackage{amsfonts}
\usepackage{amsmath}
\usepackage{amssymb}
\usepackage{wasysym}
\usepackage{yfonts}  
\usepackage{amsthm}
\usepackage{amscd}   
\usepackage[all]{xypic}
\usepackage{mathrsfs}
\usepackage{indentfirst}
\usepackage{bbold}
\usepackage{amsxtra}
\usepackage{amsbsy}
\usepackage{amstext}
\usepackage{amsopn}
\usepackage{enumerate}

\newenvironment{enumerate*}{
\begin{enumerate}[{\rm (i)}]
  \setlength{\itemsep}{3.5pt}
  \setlength{\parskip}{0pt}
}{\end{enumerate}}

\newenvironment{enumerate!}{
\begin{enumerate}[{\rm I.}]
  \setlength{\itemsep}{3.5pt}
  \setlength{\parskip}{0pt}
}{\end{enumerate}}

\CompileMatrices

\renewcommand*{\thefootnote}{\arabic{footnote}}

\newtheorem{theorem}{Theorem}[section]
\newtheorem{lemma}[theorem]{Lemma}
\newtheorem{proposition}[theorem]{Proposition}
\newtheorem{corollary}[theorem]{Corollary} 
\newtheorem{definition}[theorem]{Definition}
\newtheorem{example}[theorem]{Example}
\newtheorem{remark}[theorem]{Remark}

\newtheorem*{acknowledgements}{Acknowledgements}

\newtheorem*{thmA}{Theorem A}
\newtheorem*{corA}{Corollary A}
\newtheorem*{corB}{Corollary B}

\newtheorem*{propA}{Proposition A}

\newtheorem{conj}[theorem]{Conjecture}

\numberwithin{equation}{section}

\usepackage{titlesec}

\titleformat*{\section}{\large \bfseries}
\titleformat*{\subsection}{\it}
\titleformat*{\subsubsection}{\large\bfseries}
\titleformat*{\paragraph}{\large\bfseries}
\titleformat*{\subparagraph}{\large\bfseries}
\begin{document}
\title{Group-graded rings satisfying the strong rank condition} 
 
\author{
{\sc Peter Kropholler}\\
{\sc Karl Lorensen}
}

\maketitle
\begin{abstract} A ring $R$ satisfies the {\it strong rank condition} (SRC) if, for every natural number $n$, the free $R$-submodules of $R^n$ all have rank $\leq n$. Let $G$ be a group and $R$ a ring strongly graded by $G$ such that the base ring $R_1$ is a domain. Using an argument originated by Laurent Bartholdi for studying cellular automata, we prove that $R$ satisfies SRC if and only if $R_1$ satisfies SRC and $G$ is amenable. The special case of this result for group rings allows us to prove a characterization of amenability involving the group von Neumann algebra that was conjectured by Wolfgang L\"uck.  In addition, we include two applications to the study of group rings and their modules. 
\vspace{20pt}

\noindent {\bf Mathematics Subject Classification (2010)}:  Primary: 16W50, 16P99, 20F65, 43A07;  Secondary: 16P40, 16U20, 20F16.

\vspace{10pt}

\noindent {\bf Keywords}:  group-graded ring, strong rank condition, amenable group.
\end{abstract}
\let\thefootnote\relax

\vspace{10pt}

Following T.Y. Lam \cite{lam}, we say that a ring $R$ satisfies the {\it strong rank condition} (SRC) if, for every natural number $n$, there is no right $R$-module monomorphism $R^{n+1}\to R^n$. Replacing ``right $R$-module" with ``left $R$-module" in this definition yields the {\it left strong rank condition} (LSRC). As observed in \cite{lam}, these two conditions are {\it not} equivalent. 

This paper is concerned with the significance of SRC for the study of the property of amenability for groups. That these two properties are closely related came to light in two papers on linear cellular automata. First, T. Ceccherini-Silberstein and M. Coornaert's principal result in \cite{CSC} implied that, for a group $G$ and field $K$, the group ring $KG$ satisfies (L)SRC if $G$ is amenable. Second, in the course of proving his main result in \cite{bartholdi}, L. Bartholdi established the converse of this statement about group rings. 

In the present note, we explore
the connection between SRC and group-amenability further by examining rings more general than group rings. Our interest is in a ring $R$ that is {\it graded by a group} $G$; that is, the underlying additive group of $R$ is a direct sum of a family of abelian subgroups $\{R_g:g\in G\}$ such that $R_gR_h\subseteq R_{gh}$ for all $g,h\in G$. In such a ring, it is always the case that $R_1$ is a subring, called the {\it base ring}
of $R$ (see Proposition 1.1(i)).  If the stronger condition $R_gR_h=R_{gh}$ holds for all $g,h\in G$, then $R$ is said to be {\it strongly graded.} Thus, if $S$ is a ring and $G$ a group, then the group ring $SG$, or, more generally, any crossed product $S\ast G$, is an example of a ring strongly graded by $G$ with base ring $S$.

Our aim is to prove 

\begin{thmA} Let $G$ be a group and $R$ a ring strongly graded by $G$ such that $R_1$ is a domain.
Then the following two statements are equivalent. 
\begin{enumerate*}
\item $R$ satisfies {\rm (L)SRC}.
\item $R_1$ satisfies {\rm (L)SRC} and $G$ is amenable. 
\end{enumerate*}
\end{thmA}

We obtain the implication (ii)$\implies$(i) in Theorem A as a special case of the following more general assertion, which we prove with the aid of E. F\o lner's \cite{folner} characterization of amenability. 

\begin{propA} Let $G$ be an amenable group and $R$ a ring graded by $G$ such that, for each $g\in G$, $R_g$ contains an element
that is not a (right) left zero divisor. If $R_1$ satisfies {\rm (L)SRC}, then $R$ satisfies {\rm (L)SRC}.
\end{propA}

The implication (i)$\implies$(ii) in Theorem A is a generalization of Bartholdi's result on group rings, and we prove it by extending his argument to our setting.  

We present three applications of Theorem A, drawn from separate areas of inquiry within group theory. First, we
deduce the following characterization of amenability, conjectured by W. L\"uck [{\bf 12}, Conjecture 6.48].   
For the statement of this corollary, we employ the notation $\mathcal N(G)$ for the group von Neumann algebra of a group $G$ [{\bf 12}, Definition 1.1].

\begin{corA}
A group $G$ is amenable if and only if, for each $\mathbb C G$-module $M$ and each $p\ge1$, we have

$$\dim_{\mathcal N(G)}\left(
\mathrm{Tor}_p^{\mathbb C G}(M,\mathcal N(G))\right)=0.$$
\end{corA}

Our second corollary is 

\begin{corB} Let $G$ be a group and $S$ a domain.
If $SG$ is either right or left Noetherian, then $G$ is amenable and every subgroup of $G$ is finitely generated.
\end{corB}

Corollary B is relevant to the longstanding conjecture that, if $G$ is a group such that $\mathbb ZG$ is either left or right Noetherian, then $G$ must be virtually polycyclic. The corollary lends further credence to this conjecture, for the only amenable groups known in which every subgroup is finitely generated are those that are virtually polycyclic. 

 Our last application of Theorem A, discussed in \S 4, is to an old question about the structure of finitely generated modules over integral group rings that arises in the work of Philip Hall.  

We conclude the introduction with a comment about a possible extension of Theorem~A in the case of a group ring. 
The principal source of examples of nonamenable groups are groups that possess a noncyclic free subgroup. 
If $G$ is such a group, then it is easy to see that there is a $\mathbb ZG$-module 
monomorphism $(\mathbb ZG)^2\to \mathbb ZG$;*\footnote{*If $H$ is a free subgroup of $G$ of rank two, then the augmentation ideal $IH$ is a free $\mathbb ZH$-module of rank two, which yields a $\mathbb ZH$-module embedding $(\mathbb ZH)^2\to \mathbb ZH$. Tensoring by $\mathbb ZG$, then, gives rise to a $\mathbb ZG$-module embedding $(\mathbb ZG)^2\to \mathbb ZG$.} that is, the failure of SRC occurs already at $n=1$.  Furthermore, the same phenomenon has been observed with the integral group rings of other, more exotic, nonamenable groups (see \cite{ivanov}). As a result, it has been conjectured that this is true for the integral group rings of all nonamenable groups (see [{\bf 1}, Problem 1.1] and [{\bf 2}, Conjecture 4.1]). Formulated in the language of the present paper, the conjecture reads as follows.  

\begin{conj} The following statements are equivalent for a group $G$. 
\begin{enumerate*}
\item $\mathbb ZG$ satisfies {\rm (L)SRC}.
\item There is no (left) right $\mathbb ZG$-module embedding $(\mathbb ZG)^2\to \mathbb ZG$. 
\item $G$ is amenable. 
\end{enumerate*}
\end{conj}

With regard to (i) and (ii), it is worth pointing out that the authors are not aware of any examples of rings for which the smallest value of $n$ witnessing the failure of SRC is greater than $1$. Nevertheless, there are rings for which conditions similar to SRC fail only in higher dimensions-- for instance, {\it stable finiteness} (see \cite{lam}) and having an {\it invariant basis number} (see \cite{corner}). Hence it seems likely that such examples exist for SRC too (just, we believe, not among group rings). Finally, we remark that, for elementary reasons, any example of this sort cannot be a domain (see Proposition
1.5). 

\begin{acknowledgements}{\rm The authors began work on this paper as participants in the Research in Pairs Program of the {\it Mathematisches Forschungsinstitut Oberwolfach} from March 22 to April 11, 2015. In addition, the project was partially supported by EPSRC Grant EP/N007328/1.

Finally, the first author benefited from conversations he had in the spring of 2017 with Goulnara Arzhantseva and Denis Osin while visiting the Isaac Newton Institute at the University of Cambridge.}
\end{acknowledgements}

\section{Preliminary facts about rings}

In this section, we collect some elementary properties of rings graded by groups, as well as two propositions concerning the strong rank condition for rings. We begin by explaining some standard nomenclature regarding a ring $R$ graded by a group $G$. For each $g\in G$, the additive subgroup $R_g$ is referred to as a {\it homogeneous component}, and any element of $R_g$ is called a {\it homogeneous element}. Moreover, for any $r\in R$ and $g\in G$, we use $r_g$ to denote the image of $r$ in $R_g$ under the projection map. Also, the {\it support} of an element $r\in R$ is defined to be the set $\{g\in G:r_g\neq 0\}$. 

Our first proposition describes six basic properties of group-graded rings.  

\begin{proposition} Let $G$ be a group and $R$ a ring graded by $G$. Then the following statements hold. 
\begin{enumerate*}
\item $R_1$ is a subring of $R$.
\item $R$ is strongly graded if and only if $1\in R_gR_{g^{-1}}$ for every $g\in G$.  
\item $R$ is a crossed product if and only if each homogeneous component contains a unit. 
\item If $R$ is strongly graded, then, for any pair $g,h\in G$, there is an $(R_1,R_1)$-bimodule isomorphism $\phi_{(g,h)}:R_g\otimes_{R_1}R_h\to R_{gh}$ such that $\phi_{(g,h)}(r\otimes s):=rs$ for all $r\in R_g$ and $s\in R_h$. 

\item If $R$ is strongly graded, then, for all $g\in G$, $R_g$ is finitely generated and projective as both a right and left $R_1$-module. 
\item If $R$ is strongly graded and $R_1$ is a domain, then none of the nonzero homogeneous elements of $R$ are zero divisors.   
\end{enumerate*} 
\end{proposition}

\begin{proof} 
Statements (i)-(v) are proved in \cite{book}. To establish (vi), let $g\in G$ and $r\in R_g\backslash \{0\}$. We will show that $r$ is not a zero divisor. Let $s\in R$ such that $rs=0$. Take $t\in R_{g^{-1}}\backslash \{0\}$.
By (v), $R_g$ and $R_{g^{-1}}$ are flat, and thus also torsion-free, as $R_1$-modules. Hence, since $R_1$ is a domain, we can conclude $t\otimes r\neq 0$ in $R_{g^{-1}}\otimes_{R_1} R_g$.
Thus (iv) implies $tr\in R_1\backslash \{0\}$. But $(tr)s=0$, and $R$ is $R_1$-torsion-free. Therefore $s=0$ because $R_1$ is a domain. In other words, $r$ is not a left zero divisor. A similar argument shows that $r$ is not a right zero divisor. 
\end{proof}

A rich source of examples of strongly group-graded rings is supplied by crossed products. But there are many other instances of such rings; one is given below.

\begin{example}{\rm 
Let $S$ be a ring and suppose that $I$ is an ideal and $\phi$ is an $(S,S)$-bimodule isomorphism $I\otimes_SI\buildrel\phi\over\to S$. Let $R:=S\oplus I$ and define multiplication in $R$ by $(s,x)(s',x'):=(ss'+\phi(x\otimes x'),xs'+sx')$. This makes $R$ a ring graded by the multiplicative group $\{\pm1\}$ with $R_1=S$ and $R_{-1}=I$. For an explicit instance that is strongly graded but not a crossed product, take $S=\mathbb Z[\sqrt{-5}]$ and $I=(1+\sqrt{-5},\ 3)$. In this case, $I^2$ is the principal ideal generated by $2-\sqrt{-5}$, so we can define $\phi(x\otimes x'):=\frac{xx'}{2-\sqrt{-5}}$. The only units of $R$ are now $(\pm1,0)$, and both lie in $R_1$. Moreover, the strongly graded condition $R_{-1}R_{-1}=R_1$ is witnessed by
$\frac{3(2-\sqrt{-5})+(1+\sqrt{-5})^2}{2-\sqrt{-5}}=1.$}
\end{example}

Our proof in the next section relies on the following alternative characterization of SRC given in \cite{lam}.

\begin{proposition} The following statements are equivalent for a ring $R$. 
\begin{enumerate*}
\item $R$ satisfies {\rm SRC}.
\item Any system of equations

\begin{align*} a_{11}x_1+a_{12}x_2+&\cdots+a_{1n}x_n=0  \\
a_{21}x_1+a_{22}x_2+&\cdots+a_{2n}x_n=0  \\
&\vdots \\
a_{m1}x_1+a_{m2}x_2+&\cdots+a_{mn}x_n=0
\end{align*}
over $R$ such that $m<n$ has a nonzero solution for the unknowns $x_1,\dots, x_n$.  \hfill\(\square\)
\end{enumerate*}  

\end{proposition}

\begin{remark}{\rm  There is also a version of Proposition 1.3 for LSRC in which the coefficients in the equations appear on the right.} 
\end{remark}

For a domain, there are four other important conditions that are equivalent to SRC, proved in \cite{lam}. 
We will not invoke these in our proofs, but they are interesting to keep in mind, especially in view of the 
relevance of the second property to Conjecture 0.1. 

\begin{proposition} The following statements are equivalent for a domain $R$. 

\begin{enumerate*}
\item $R$ satisfies {\rm (L)SRC}.
\item There is no (left) right $R$-module monomorphism $R^2\to R$.
\item $R$ is a (left) right Ore ring.  
\item $R$ has finite uniform dimension as a (left) right $R$-module.
\item $R$ has uniform dimension $1$ as a (left) right $R$-module. \hfill\(\square\)
\end{enumerate*}  
\end{proposition}

Recall that the {\it uniform dimension} of a module is the maximal number of summands in any submodule that can be expressed as a direct sum (see \cite{lam}). 

\section{Proof of Theorem A((ii)$\implies$(i))}

This section is devoted to the proof of Proposition A; in view of Proposition 1.1(vi), the implication (ii)$\implies$(i) in Theorem A will then follow immediately. First we state F\o lner's \cite{folner} well-known characterization of amenability, which will serve as the basis of our proofs here and in the next section.

\begin{proposition}{\rm (F\o lner)} A group $G$ is amenable if and only if, for every finite subset $S$ of $G$ and every $\epsilon>0$, there is a finite subset $F$ of $G$ such that 
\[
\pushQED{\qed}
|SF|<(1+\epsilon)|F|.\qedhere
\popQED
\]
\end{proposition}

\begin{proof}[Proof of Proposition A] We will confine ourselves to the part of the implication involving SRC, the left condition requiring merely a dual version of this argument. Our appoach is to invoke Proposition 1.3. To this end, let

\begin{equation} \sum_{j=1}^n a_{ij}x_j=0\ \ \ \ \ \ \ i=1,\dots, m\end{equation}
be a system of $m$ equations over $R$ with $n$ unknowns $x_1,\dots,x_n$ such that $m<n$.  Exploiting the fact that $R_1$ satisfies SRC, we will produce values of the $x_j$, not all equal to zero, that satisfy (2.1). 

We begin by taking, for each $g\in G$,  an element $b_g\in R_g$ such that $b_g$ is not a left zero divisor. Also, let $S$ be the union of the supports of the coefficients $a_{ij}$. Invoking Proposition 2.1, we can find a finite set $F\subseteq G$ such that $|SF|<\frac{n}{m}|F|$. 
Consider now the linear system

\begin{equation} \sum_{f\in F\ j=1,\dots,n}[b_{g^{-1}}(a_{ij})_{gf^{-1}}b_f]x'_{jf}=0 \ \ \  \mbox{for}\ g\in SF,\ \ i=1,\dots,m       \end{equation}
with $m|SF|$ equations and $n|F|$ unknowns $x'_{jf}$. Since the coefficients $b_{g^{-1}}(a_{ij})_{gf^{-1}}b_f$ are all elements of $R_1$ and $m|SF|<n|F|$, there must be elements $x'_{jf}$ of $R_1$, not all zero, that satisfy (2.2).  

For each $j=1,\dots,n$, set $x_j:=\sum_{f\in F}b_fx'_{jf}$. Then at least one of the $x_j$ is nonzero. Moreover, for every $g\in SF$ and $i=1,\dots,m$,

$$\sum_{f\in F\ j=1,\dots,n}b_{g^{-1}}(a_{ij})_{gf^{-1}}(x_j)_f=0,$$ 
implying

$$\sum_{f\in F\ j=1,\dots,n}(a_{ij})_{gf^{-1}}(x_j)_f=0.$$ In other words,
$\sum_{j=1}^n (a_{ij}x_j)_g=0$
for every $g\in SF$ and $i=1,\dots,m$. Hence, since $SF$ contains the union of the supports of the left sides of all the equations in (2.1), we see that the equations hold. 
\end{proof}

\begin{remark}{\rm The special case of the above argument for $(m,n)=(1,2)$ and $R=KG$ for $K$ a field is well known and appears in the discussion of [{\bf 1}, Problem 1.1].  It is also contained in D. Kielak's appendix to Bartholdi's paper \cite{bartholdi}, where it is attributed to D. Tamari \cite{tamari}.}
\end{remark}

That the hypothesis about the homogeneous components in Proposition A is necessary is illustrated by

\begin{example}{\rm  Let $S$ be any ring such that $S$ fails to satisfy SRC and the canonical ring homomorphism $\mathbb Z\to S$ is injective. For example, $S$ could be the integral group ring of a free group. Let $R$ be the ring graded by the integers with the following properties:

\begin{enumerate*}

\item $R_k=S$ for $k\in \mathbb Z^+$, $R_0=\mathbb Z$, and $R_k=0$ for $k\in \mathbb Z^-$.

\item If $k\in \mathbb Z^+$, the product maps $R_0\times R_k\to R_{k}$ and $R_k\times R_0\to R_{k}$ are defined using the left and right $\mathbb Z$-module structures on $S$. 

\item If $k, l\in \mathbb Z^+$, the product map $R_k\times R_l\to R_{k+l}$ is defined using multiplication in the ring $S$.
\end{enumerate*}

\noindent Notice that $R$ is therefore isomorphic to the subring of the polynomial ring $S[x]$ consisting of all those polynomials with an integer constant term. 

Consider a homogeneous linear system over $S$ with fewer equations than unknowns that has only the trivial solution. Treating the coefficients as elements of $R_1$, we can also regard this system as one over $R$. From this point of view too, the system has only the trivial solution, as there are no nontrivial solutions in $R_k$ for each  
$k\in \mathbb Z$. Therefore $R$ fails to fulfill SRC, even though it is satisfied by $R_0$. }
\end{example}

\section{Proofs of Theorem A((i)$\implies$(ii)) and Corollary A}

The proof of the second implication in Theorem A is based on the following three lemmas extracted from the proof of [{\bf 3}, Theorem 1.1]. The first of these is established by [{\bf 3}, Equation (2.3)] and the preceding paragraph, using the F\o lner condition. 

\begin{lemma}{\rm (Bartholdi)} If a group $G$ is not amenable, then there is a finite set $S\subseteq G$ such that, for every finite subset $F$ of $G$,

\[
\pushQED{\qed}
|SF|>(1+\log |S|)|F|.\qedhere
\popQED
\]
\end{lemma}

Lemma 3.2 below is a consequence of [{\bf 3}, Lemma 2.1] and its proof, as well as the reasoning presented in the paragraph following [{\bf 3}, Equation (2.3)].

\begin{lemma}{\rm (Bartholdi)} Let $S$ be a finite set with $|S|\geq 2$. Then there exist a finite set $Y$ and a family of subsets $\{X_s:s\in S\}$ of $Y$ such that the two assertions below are true.

\begin{enumerate*}

\item $\displaystyle{\Big |\bigcup_{s\in S} X_s\Big |=|Y|-1.}$

\item For every subset $T$ of $S$,

$$\Big |X_s\backslash \bigcup_{t\in T\backslash \{s\}} X_t\Big |\geq \frac{|Y|}{(1+\log |S|)|T|}$$
for all $s\in T$. \qed

\end{enumerate*}

\end{lemma}

With regard to Lemma 3.2(i), the reader will notice that Bartholdi observes merely that $\bigcup_{s\in S} X_s\subsetneq Y$.
However, a look at his construction of the sets $Y$ and $X_s$ reveals that our stronger assertion holds.

The last of the lemmas gleaned from the proof of [{\bf 3}, Theorem 1.1] is

\begin{lemma}{\rm (Bartholdi)} Take $K$ to be a field. Let $S$ and $Y$ be finite sets and $\{X_s:s\in S\}$ a family of subsets of $Y$. 
Moreover, for each $s\in S$ and $T\subseteq S$, write 

$$X_{s,T}:=X_s\backslash \bigcup_{t\in T\backslash \{s\}} X_t.$$
Then there is a finite field extension $L$ of $K$ and there are linear maps $\alpha_s:LY\to LY$ for all $s\in S$ such that the following two properties hold.
\begin{enumerate*}

\item ${\rm Im}\ \alpha_s\subseteq LX_s$ for all $s\in S$.
\item Whenever  $\{T_s:s\in S\}$ is a family of subsets of $S$ with $\sum_{s\in S}|X_{s,T_s}|\geq |Y|$, we have

$$\bigcap_{s\in S} {\rm Ker}\ \alpha_{s,T_s}=0,$$
where $\alpha_{s,T_s}$ is the composition of the projection map $\pi_{s,T_s}:X_s\to X_{s,T_s}$ with $\alpha_s$.
\end{enumerate*}

\end{lemma}

We furnish a proof of Lemma 3.3 that employs exactly the same reasoning as \cite{bartholdi} but includes more detail. The argument relies on the following simple observation.

\begin{lemma} Let $K$ be a field and $f(x_1,\dots,x_n)$ a nonconstant polynomial in $n$ variables whose coefficients lie in $K$. Then, for any $b\in K$, there is a finite field extension $L$ of $K$ containing elements $a_1, \dots, a_n$ such that
$f(a_1,\dots, a_n)=b$. 
\end{lemma}

\begin{proof} We proceed by induction on $n$. If $n=1$, we take $L$ to be any finite field extension of $K$ such that the polynomial $f(x_1)-b$ has a zero in $L$. Now we consider the case $n>1$.  If $x_n$ fails to occur in any nonzero term in $f(x_1,\dots, x_n)$, then the conclusion follows from the inductive hypothesis. Assume, then, that $x_n$ occurs in some nonzero term. Regard $f(x_1,\dots,x_n)$ as a polynomial in the variable $x_n$ with coefficients in $K[x_1,\dots, x_{n-1}]$. At least one of the coefficients above the constant term in this polynomial is nonzero; denote that coefficient by $g(x_1,\dots,x_{n-1})$. By the inductive hypothesis, there is a finite field extension $L_0$ of $K$
in which there are elements $a_1, \dots, a_{n-1}$ such that $g(a_1,\dots,a_{n-1})\neq 0$. Hence $f(a_1,\dots,a_{n-1}, x_n)$ is a nonconstant polynomial in $L_0[x_n]$. Thus there is field $L$ containing $L_0$ such that $[L:L_0]<\infty$ and $L$ has an element $a_n$ with $f(a_1,\dots, a_n)=b$.

\end{proof}

\begin{proof}[Proof of Lemma 3.3] Write $m:=|Y|$; for convenience, we will identify $Y$ with the set $\{1,\dots,m\}$.
 For each $s\in S$, let $A_s$ be the $m\times m$ matrix with the following two properties:
\begin{itemize}
\item For every $i\in X_s$ and $j=1,\dots, m$, the $(i,j)$-entry of $A_s$ is the indeterminate $x_{sij}$.
\item  For every $i\in Y\backslash X_s$, the $i$th row of $A_s$ is comprised entirely of zeros.
\end{itemize}

For every family $\mathcal{T}:=\{T_s:s\in S\}$ of subsets of $S$ with $v_{\mathcal{T}}:=\sum_{s\in S}|X_{s,T_s}|\geq m$, let $B_{\mathcal{T}}$ be the $v_{\mathcal{T}}\times m$ matrix formed by taking from each matrix $A_s$ the rows corresponding to the elements of $X_{s,T_s}$. In constructing the matrices $B_{\mathcal{T}}$, we proceed according to a specific order assigned to $S$ and also follow the increasing order on $X_{s,T_s}$. 
Next let $D_{\mathcal{T}}$ be the determinant of the first $m$ rows of $B_{\mathcal{T}}$. Then  $D_{\mathcal{T}}$ is a nonzero polynomial in the indeterminates $\{x_{sij}:s\in S, i\in X_s, 1\leq j\leq m\}$. 

According to Lemma 3.4, there is a finite field extension $L$ of $K$ containing a subset $\{a_{sij}:s\in S, i\in X_s, 1\leq j\leq m\}$ such that, if the $a_{sij}$ are substituted for the variables $x_{sij}$, the product $\prod_{\mathcal{T}} D_{\mathcal{T}}$ is nonzero. Note that this product is taken over all possible families $\mathcal{T}:=\{T_s:s\in S\}$ of subsets of $S$ with $v_{\mathcal{T}}\geq m$. Next, for each $s\in S$, let $A'_s$ be the matrix obtained from $A_s$ by replacing the indeterminate $x_{sij}$ by the value $a_{sij}$ for every $i\in X_s$ and $j=1,\dots m$. 

Finally, for each $s\in S$, let $\alpha_s$ be the linear map $LY\to LY$ whose matrix is $A'_s$ . Then the maps $\alpha_s$ plainly fulfill condition (i). Moreover, condition (ii) follows from the fact that, for any family $\mathcal{T}:=\{T_s:s\in S\}$ of subsets of $S$ with $v_{\mathcal{T}}\geq m$,  the polynomial $D_{\mathcal{T}}$ has a nonzero value if each $x_{sij}$ assumes the value $a_{sij}$. 
\end{proof}

We are now ready to apply Bartholdi's ideas to group-graded rings.

\begin{proof}[Proof of Theorem A((i)$\implies$(ii))] We will just prove the implication for SRC, the LSRC case being similar. Our approach is to establish the contrapositive, so we assume that either $G$ is not amenable or $R_1$ does not satisfy SRC. 
In the second case, there is a right $R_1$-module embedding $\iota :R_1^{n+1}\to R_1^n$ for some $n\in \mathbb N$. Proposition 1.1(v) implies that $R$ is a flat $R_1$-module, which means that tensoring $\iota$ by the identity map $R\to R$ produces a right $R$-module embedding $R^{n+1}\to R^n$. Therefore $R$ fails to satisfy SRC. 

Next we treat the case where $G$ is not amenable. Let $Z$ denote the commutative subring of $R_1$ consisting of all the integer multiples of $1$, and let $K$ be the field of fractions of $Z$. As a flat $R_1$-module, $R$ is $R_1$-torsion-free, which means that $R$ embeds in the ring $R':=R\otimes_ZK$. It is then easy to see that any right $R'$-module embedding $(R')^{n+1}\to (R')^n$ gives rise to a right $R$-module embedding $R^{n+1}\to R^n$. Hence there is no real loss of generality in assuming $K\subseteq R_1$.  

According to Lemma 3.1, there is a finite subset $S$ of $G$ such that 
$|SF|>(1+\log |S|)|F|$ for all finite sets $F\subseteq G$. Choose next a finite set $Y$ and a family of subsets $\{X_s:s\in S\}$ that have the properties described in Lemma 3.2. 
Furthermore, take a field $L$ and linear maps $\alpha_s:LY\to LY$ for every $s\in S$ to be as in Lemma 3.3. In addition, for each $s\in S$, pick a nonzero element $b_s\in R_s$, and let $\beta_s:R\to R$ be the right $R$-module monomorphism $r\mapsto b_sr$ (see Proposition 1.1(vi)). Finally, define the right $L\otimes_K R$-module homomorphism $\Theta: LY\otimes_K R\to LY\otimes_K R$ by 
$\Theta:=\sum_{s\in S} (\alpha_s\otimes \beta_s).$
Notice that ${\rm Im}\ \Theta$ is contained in a free $L\otimes_K R$-submodule of $LY\otimes_K R$ of rank $|Y|-1$.  We will now show that $\Theta$ is injective. This will mean that there is an $R$-module embedding $R^{|Y|[L:K]}\to R^{(|Y|-1)[L:K]}$, thus demonstrating that $R$ fails to satisfy SRC.  

To prove that $\Theta$ is monic, let $u\in LY\otimes_K R$ with $u\neq 0$. Since $R$ is graded by $G$, we can view $LY\otimes_K R$ as a ring graded by $G$. From this perspective, let $F\subseteq G$ be the support of $u$.  As argued in the final paragraph of the proof of [{\bf 3}, Theorem 1.1], it follows from Lemma 3.2 that there is an $f_0\in F$ such that, if $T_s:=\{t\in S:sf_0\in tF\}$ for each $s\in S$, we have $\sum_{s\in S}|X_{s,T_s}|\geq |Y|$.  Consequently, by Lemma 3.3(ii), $\bigcap_{s\in S} {\rm Ker}\ \alpha_{s,T_s}=0.$ Thus 
$\bigcap_{s\in S} {\rm Ker}(\alpha_{s,T_s}\otimes \beta_s)=0$, which implies $(\alpha_{s_0,T_{s_0}}\otimes \beta_{s_0})(u_{f_0})\neq 0$ for some $s_0\in S$. Now we consider the $LX_{s_0,T_{s_0}}\otimes_K R_{s_0f_0}$-component of $\Theta(u)$. Writing $F':=\{f\in F:f\neq f_0\ \mbox{and}\ s_0f_0f^{-1}\in S\}$, we can express this component as 

\begin{equation} (\alpha_{s_0,T_{s_0}}\otimes \beta_{s_0})(u_{f_0}) + \sum_{f\in F'}(\pi_{s_0,T_{s_0}}\alpha_{s_0f_0f^{-1}}\otimes \beta_{s_0f_0f^{-1}})(u_f).\end{equation}

\noindent If $f\in F'$, then $s_0f_0f^{-1}\in T_{s_0}-\{s_0\}$, which means $\pi_{s_0,T_{s_0}}\alpha_{s_0f_0f^{-1}}=0$. Hence all the terms in the summation on the right in (3.1) are zero. Therefore $\Theta(u)\neq 0$, so that $\Theta$ is injective. 
\end{proof}

\begin{remark} {\rm  For the benefit of the reader, we clarify precisely where the two hypotheses about the graded ring $R$ play a role in the proof. The fact that the grading is strong is
necessary to conclude that $R$ is $R_1$-flat, which is invoked in the first and second paragraphs. That $R_1$ is a domain is also employed in two places. First, in the second paragraph, we use that $Z$ fails to contain any nonzero zero divisors of $R_1$. Second, in the third paragraph, we need that $R_1$ is a domain to apply Proposition 1.1(vi) to obtain that $\beta_s$ is injective. 
}
\end{remark}

Corollary A is proved very easily from Theorem A((i)$\implies$(ii)). 

\begin{proof}[Proof of Corollary A]
The {\it only if} direction is proved in [{\bf 12}, Theorem 6.37]. For the {\it if} statement, it will be convenient to prove the contrapositive. Suppose that $G$ is non-amenable. Theorem A yields that, for some $n\in \mathbb N$, there is a right $\mathbb CG$-module embedding  
$\iota:(\mathbb C G)^{n+1}\to (\mathbb C G)^{n}.$
Tensoring with $\mathcal N(G)$, we obtain an exact sequence
\begin{equation*} 0\to \mathrm{Tor}_1^{\mathbb C G}(M,\mathcal N(G))\to\mathcal N(G)^{n+1}\to\mathcal N(G)^{n},\end{equation*} 
where $M={\rm coker}\ \iota$. Hence $\dim_{\mathcal N(G)}(\mathrm{Tor}_1^{\mathbb C G}(M,\mathcal N(G)))\geq1$.  
\end{proof}

\section{Corollary B and the Hall property}

In our final section, we prove Corollary B and discuss its implications for the study of the structure of the underlying additive group of a finitely generated module over an integral group ring.  

\begin{proof}[Proof of Corollary B] We will just consider the case where $SG$ is right Noetherian; the other case follows by a similar argument.
Since right Noetherian rings satisfy SRC, we have that $G$ must be amenable by Theorem A. To verify the second condition, we will construct an injective map $H\mapsto I_H$ from the set of subgroups of $G$ to the set of right ideals of $SG$ that preserves inclusion. Let $H$ be a subgroup of $G$, and choose $\{x_j:j\in J\}$ to be a complete set of right coset representatives of $H$ in $G$. For every element $r\in SG$, we write $r_g\in S$ for the coefficient of $g$ in the sum $r$. Using this notation, we set

$$I_H:=\{r\in SG:\sum_{g\in Hx_j} r_g=0\ \mbox{for all $j\in J$}\}.$$

\noindent 
Notice that $I_H$ is a right ideal of $SG$, and that the map
$H\mapsto I_H$ has the desired properties. As a result, since $SG$ is right Noetherian, $G$ must satisfy the maximal condition on subgroups; in other words, every subgroup of $G$ is finitely generated. 
\end{proof}

\begin{remark} {\rm The argument that $SG$ being Noetherian implies that every subgroup of $G$ is finitely generated is well known; it appears, for example, in a contribution of Y. de Cornulier \cite{cornulier} on {\it mathoverflow.net}. } 

\end{remark}

The reader might ask whether Corollary B can be generalized to strongly group-graded rings that are Noetherian. The answer is that, although the amenability conclusion clearly follows from Theorem A, the group may have non-finitely-generated subgroups, as the following example illustrates. 

\begin{example}{\rm
Let $p$ be a prime number, and let $K$ be the subfield of $\mathbb R$ generated by $\{p^{1/n}:\ n\in \mathbb N\}$. Being a field, $K$ is Noetherian. Now take $G$ to be the additive abelian group $\mathbb Q/\mathbb Z$. Then $K$ is strongly graded by $G$ with $K_{r+\mathbb Z}:=\mathbb Qp^r$ for each $r\in \mathbb Q$. In fact, $K$ is a crossed product $\mathbb Q\ast G$. However, $G$ is not finitely generated. }
\end{example}

Now we use Corollary B to shed some light on a question suggested by an important result about modules over polycyclic groups proved by Philip Hall.
In [{\bf 8}, Lemma 5.2], he shows that, if $G$ is a polycyclic group, then, for any finitely generated $\mathbb ZG$-module $M$, there is a finite set of primes $\pi$ such that the underlying additive group of $M$ is an extension of a free abelian group by a $\pi$-torsion group. Inspired by Hall's lemma, we make the following definition.

\begin{definition} {\rm A group $G$ has the {\it Hall property} if, for every finitely generated $\mathbb ZG$-module $M$, there is a finite set of primes $\pi$ such that the underlying additive group of $M$ is an extension of a free abelian group by a $\pi$-torsion group.} 
\end{definition}

In its most general form, Hall's result may be stated as follows. 

\begin{theorem}{\rm (Hall) [{\bf 11}, 4.3.3]} Every virtually polycyclic group has the Hall property. 
\end{theorem}

One consequence of Theorem 4.4 is that every simple $\mathbb ZG$-module has prime characteristic if $G$ is virtually polycyclic; this is the first step toward proving that such modules are actually finite, established by J. E. Roseblade \cite{roseblade}. It is a longstanding question whether there are any other sorts of groups whose simple modules also all have prime characteristic. A significant advance on this problem was accomplished by A.V. Tushev [{\bf 16}, Theorem 3], who showed that this is true for virtually metabelian groups of finite rank. In order to make further progress, it would be beneficial to identify some other types of groups that satisfy the Hall property. Unfortunately, however, as we show in our next theorem, the class of groups enjoying this property is, in fact, severely limited. 

\begin{theorem} Let $G$ be a group with the Hall property. Then $\mathbb QG$ is right and left Noetherian; hence, by Corollary B, $G$ is amenable, and every subgroup of $G$ is finitely generated.

\end{theorem}

Virtually solvable groups whose subgroups are all finitely generated must be virtually polycyclic. Hence Theorem 4.5 yields 

\begin{corollary} A virtually solvable group has the Hall property if and only if it is virtually polycyclic. 
\hfill\(\square\)
\end{corollary}

In order to prove Theorem 4.5, we require the following lemma. 

\begin{lemma}
Let $A$ be a torsion-free abelian group (written additively) which is (free abelian)-by-($\pi$-torsion) for some set of primes $\pi$. Let $B$ be a subgroup such that $A/B$ is torsion-free. Then, for any prime $q\notin \pi$, we have $B\nleq qA$.
\end{lemma}
\begin{proof}
Let $F$ be a free abelian subgroup of $A$ such that $A/F$ is $\pi$-torsion. 
Then $B\cap F$ is a free abelian subgroup of $B$, and $B/B\cap F$ is $\pi$-torsion.
For any prime $q\notin \pi$, we have $qB\cap F=q(B\cap F)\neq B\cap F$. Therefore, since $qA\cap B=qB$, we conclude $qA\cap B\neq B$.
\end{proof}

\begin{proof}[Proof of Theorem 4.5] 
Suppose that $\mathbb QG$ is not right Noetherian. Then there is a strictly ascending chain $I_1<I_2<I_3<\dots$ of right ideals in $\mathbb Q G$. 
Set $J_n:=I_n\cap \mathbb Z G$. Since $G$ satisfies the Hall property, $\mathbb ZG/J_n$ is (free abelian)-by-($\pi_n$-torsion) for some finite set of primes $\pi_n$. Applying Lemma 4.7 with  $A:=\mathbb Z G/J_{n}$ and $B:=J_{n+1}/J_n$, we deduce that, for each $n\in \mathbb N$,  there are infinitely many primes $q$ such that $J_{n+1}\not\subseteq J_n+q\mathbb Z G$. As a result, we can inductively define a sequence of \emph{distinct} primes $q_n$ such that  $J_{n+1}\not\subseteq J_n+q_n\mathbb Z G$ for all $n\in \mathbb N$.

For each $n\in \mathbb N$, choose $\xi_n \in J_{n+1}$ so that $\xi_n\notin J_n+q_n\mathbb Z G$. Writing $K_n:=J_n+q_n\mathbb Z G$, take $M$ to be the submodule of $\displaystyle{\prod_{n=1}^{\infty}({\mathbb Z G}/{K_n})}$ generated by the sequence $\mathbf{1}:=(1+K_1,1+K_2,1+K_3,\dots)$. Then the $n$th entry of $\mathbf{1}\xi_n$ is non-zero and annihilated by multiplication by $q_n$, whereas the $(n+1)$st and all subsequent entries are zero. We infer that $M$ has elements of order $q_n$ for all $n$ and so is witness to the failure of the Hall property, a contradiction. Therefore $\mathbb QG$ must be right Noetherian. That $\mathbb QG$ is also left Noetherian may be shown by a dual argument. 
\end{proof}


\begin{thebibliography}{16}


\bibitem[{\bf 1}]{problem} Preliminary problem list of the workshop ``Thompson's Group at 40 Years" (January  11-14, 2004, American Institute of Mathematics, Palo Alto, California). Available at http://aimath.org/WWN/thompsonsgroup/thompsonsgroup.pdf.


\bibitem[{\bf 2}]{bartholdi2}{\sc L. Bartholdi.}  Gardens of Eden and amenability on cellular automata. {\it J.  Eur. Math. Soc.} {\bf 12} (2010), 241-248.


\bibitem[{\bf 3}]{bartholdi}{\sc L. Bartholdi} (with an appendix by {\sc D. Kielak}). Amenability of groups is characterized by Myhill's theorem.  To appear in {\it J.  Eur. Math. Soc.} (2019). Available at arXiv:1605.09133v2.

\bibitem[{\bf 4}]{CSC}{\sc T. Ceccherini-Silberstein} and {\sc M. Coornaert.} The Garden of Eden theorem for linear cellular
automata. {\it Ergodic Theory Dynam. Systems} {\bf 26} (2006), 53-68.




\bibitem[{\bf 5}]{corner}{\sc A. L. S. Corner.} Additive categories and a theorem of W. G. Leavitt. {\it  Bull. Amer. Math. Soc.} {\bf 75} (1969), 78--82.


\bibitem[{\bf 6}]{cornulier}{\sc Y. de Cornulier.} {\it mathoverflow.net/questions/165039}. 


\bibitem[{\bf 7}]{folner}{\sc E. F\o lner.} On groups with full Banach mean value. {\it Math. Scand.} {\bf 3} (1955), 243-254. 

\bibitem[{\bf 8}]{hall}{\sc P. Hall.} On the finiteness of certain soluble groups. {\it Proc. London Math. Soc.} {\bf 20} (1959), 595-622. 

 \bibitem[{\bf 9}]{ivanov}{\sc S. V. Ivanov}. Group rings of Noetherian groups. (Russian) {\it Mat. Zametki} {\bf 46} (1989), 61-66; translation in {\it Math. Notes} {\bf 46} (1989), 929-933.


\bibitem[{\bf 10}]{lam}{\sc T. Y. Lam.} {\it Lectures on Modules and Rings} (Springer, 1996).

\bibitem[{\bf 11}]{robinson}{\sc J. C. Lennox} and {\sc D. J. S. Robinson}. {\it The Theory of Infinite Soluble Groups} (Oxford, 2004).


\bibitem[{\bf 12}]{lueck}{\sc W. L\"uck.} {\it $L^2$-Invariants: Theory and Applications to Geometry and $K$-Theory} (Springer, 2002).

\bibitem[{\bf 13}]{book}{\sc C. N\v{a}st\v{a}sescu} and {\sc F. Van Oystaeyen}. {\it Methods of Graded Rings}, Lecture Notes in Math. {\bf 1836} (Springer, 2004).


\bibitem[{\bf 14}]{roseblade}{\sc J. E. Roseblade.} Group rings of polycyclic groups. {\it J. Pure Appl. Algebra.} {\bf 3} (1973), 307-328. 

\bibitem[{\bf 15}]{tamari}{\sc D. Tamari.} A refined classification of semi-groups leading to generalised polynomial rings
with a generalised degree concept. {\it Proc. ICM, vol. 3} (Amsterdam, 1954), 439-440.



\bibitem[{\bf 16}]{tushev}{\sc A. V. Tushev.} On solvable groups with proper quotient groups of finite rank. {\it  Ukrain. Math. J.} {\bf 54} (2002), 1897-1905. 


\end{thebibliography}
\end{document}